\documentclass[11pt]{amsart}
\usepackage[margin=1in]{geometry}

\usepackage{MnSymbol}
\usepackage{mathrsfs}
\usepackage{colonequals}

\usepackage{hyperref}
\hypersetup{
pdflang={en-US},
pdftitle={The normal bundle of a general canonical curve of genus at least 7 is semistable},
pdfauthor={Izzet Coskun, Eric Larson, and Isabel Vogt},
pdfsubject={Algebraic Geometry},
pdfkeywords={canonical curves, normal bundle, stability}
}

\usepackage{tikz-cd}
\usetikzlibrary{decorations.pathreplacing}

\makeatletter
\newcommand{\oset}[3][0ex]{%
  \mathrel{\mathop{#3}\limits^{
    \vbox to#1{\kern-2\ex@
    \hbox{\(\scriptstyle#2\)}\vss}}}}
\makeatother

\newcommand{\pp}{\mathbb{P}}

\newcommand{\N}{\mathcal{N}}
\renewcommand{\O}{\mathcal{O}}

\newcommand{\sC}{\mathscr{C}}
\newcommand{\sF}{\mathscr{F}}
\newcommand{\sI}{\mathscr{I}}
\newcommand{\sV}{\mathscr{V}}

\newcommand{\sing}{\text{sing}}
\newcommand{\Pic}{\operatorname{Pic}}
\newcommand{\Bl}{\operatorname{Bl}}
\renewcommand{\Im}{\operatorname{Im}}
\newcommand{\codim}{\operatorname{codim}}
\newcommand{\rk}{\operatorname{rk}}
\newcommand{\ev}{\operatorname{ev}}
\renewcommand{\bar}{\overline}
\renewcommand{\hat}{\widehat}

\newcommand{\leqpar}{\underset{{\scriptscriptstyle (}-{\scriptscriptstyle )}}{<}}

\newcommand{\posmod}{\oset{+}{\rightarrow}}
\newcommand{\negmod}{\oset{-}{\rightarrow}}
\newcommand{\posmodalong}{\oset{+}{\leadsto}}
\newcommand{\negmodalong}{\oset{-}{\leadsto}}

\newcommand{\qand}{\quad \text{and} \quad}
\newcommand{\qwhere}{\quad \text{where} \quad}

\newcommand{\defi}[1]{\textsf{#1}} 

\newtheorem{thm}{Theorem}[section]
\newtheorem{lem}[thm]{Lemma}
\newtheorem{prop}[thm]{Proposition}
\newtheorem{cor}[thm]{Corollary}

\theoremstyle{definition}
\newtheorem{defin}[thm]{Definition}
\newtheorem{cond}[thm]{Condition}
\newtheorem{setup}[thm]{Setup}

\theoremstyle{remark}
\newtheorem{rem}{Remark}

\title{The normal bundle of a general canonical curve of genus at least \(7\) is semistable}
\author{Izzet Coskun}
\address{Department of Mathematics, Statistics, and CS \\
University of Illinois at Chicago, Chicago IL 60607}
\email{icoskun@uic.edu}
\author{Eric Larson}
\address{Department of Mathematics, Brown University}
\email{elarson3@gmail.com}
\author{Isabel Vogt}
\address{Department of Mathematics, Brown University}
\email{ivogt.math@gmail.com}

\thanks{During the preparation of this article, I.C.\ was supported
by NSF FRG grant DMS-1664296 and NSF grant DMS-2200684,  E.L. was supported by
NSF  grants DMS-1802908 and DMS-2200641, and I.V. was supported by NSF grants DMS-1902743 and DMS-2200655.}

\begin{document}

\begin{abstract}
Let \(C\) be a general canonical curve of genus \(g\) defined over an algebraically closed field of arbitrary characteristic. We prove that if \(g \notin \{4,6\}\), then the normal bundle of \(C\) is semistable.
In particular, if \(g \equiv 1\) or \(3 \pmod{6}\), then the normal bundle is stable.     
    
\end{abstract}
\maketitle

\section{Introduction}

Let \(k\) be an algebraically closed field of arbitrary characteristic. Let \(C\) be a nonsingular, irreducible, non-hyperelliptic curve of genus \(g \geq 3\) defined over \(k\). Then the canonical linear system \(|K_C|\) embeds \(C\) in \(\pp^{g-1}\). The image is called a  \defi{canonical curve of genus \(g\)}. Canonical curves of genus \(g\) lie in an irreducible component of the Hilbert scheme of curves of genus \(g\) in \(\pp^{g-1}\). Studying the properties of canonical curves is an essential tool in curve theory. 

Given a vector bundle \(V\) on \(C\) of rank \(r\) and degree \(d\), recall that the \defi{slope} of \(V\) is defined by \(\mu(V)\colonequals \frac{d}{r}\). The  bundle \(V\) is called \defi{semistable} if, for every proper subbundle \(W\), we have \(\mu(W) \leq \mu(V)\). The bundle is called \defi{stable} if the inequality is always strict. 

Since stable bundles are the atomic building blocks of all vector bundles on a curve, it is important to ask if  naturally-defined vector bundles on canonical curves, such as the restricted tangent bundle \(T_{\pp^{g-1}}|_C\) or the normal bundle \(N_C\), are stable.  The first of the these is straightforward:  the restricted tangent bundle of a general canonical curve of genus \(g \geq 3\) is always stable.  In fact, the restricted tangent bundle of a general Brill--Noether curve of any degree \(d\) and genus \(g \geq 2\) in \(\pp^r\) is stable unless \((d,g) = (2r, 2)\) \cite{fl}.  On the other hand, the normal bundle can fail to be stable in low genus (cf.~ Remark  \ref{rem:counter}).

Aprodu, Farkas and Ortega \cite{AFO} conjectured that once the genus is sufficiently large, the normal bundle of a general canonical curve is stable. 
Previously, this was only known for \(g=7\) \cite{AFO} and for \(g=8\) \cite{Bruns}.  The proofs of these two results use explicit models of low genus canonical curves due to Mukai, and  thus do not generalize to large genus.
In this paper, we prove:

\begin{thm}\label{thm:main}
Let \(C\) be a general canonical curve of genus \(g \not\in \{4,6\}\) defined over an algebraically closed field of arbitrary characteristic. Then the normal bundle of \(C\) is semistable. 
\end{thm}

\noindent The rank
of \(N_C\) is \(g-2\) and the degree of \(N_C\) is \(2(g^2-1)\). Hence, 
\[\mu(N_C) = 2g+4 + \frac{6}{g-2}.\]
In particular, if \(g-2\) and \(6\) are relatively prime, the semistability of \(N_C\) implies the stability of \(N_C\). We thus obtain the following corollary.

\begin{cor}
If \(g \equiv  1\) or \(3 \pmod{6}\), then the normal bundle of the general canonical curve of genus \(g\) is stable.
\end{cor}

\begin{rem}\label{rem:counter}
When \(g=3\),  the canonical curve is a plane quartic curve. Hence, \(N_C \cong \O_C(4)\) and is stable. When \(g=5\), the general canonical curve is a complete intersection of three quadrics. Hence, \(N_{C}\cong \O_C(2)^{\oplus 3}\). In particular, \(N_C\) is semistable but not stable. When \(g=4\) or \(6\), \(N_C\) is unstable, as we now explain. When \(g=4\), the canonical curve is a complete intersection of type \((2,3)\). The normal bundle of \(C\) in the quadric is a destabilizing line subbundle of \(N_C\) of degree \(18\). When \(g=6\), the general canonical curve is a quadric section of a quintic del Pezzo surface. The normal bundle of \(C\) in this del Pezzo surface gives a degree 20 destabilizing line subbundle of \(N_C\).
\end{rem}

We will prove Theorem \ref{thm:main} by specializing a canonical curve to the union of an elliptic normal curve \(E\) of degree \(g\) and a \(g\)-secant rational curve \(R\) of degree \(g-2\) meeting \(E\) quasi-transversely in \(g\) points.
In \S \ref{sec:our_degen}, we describe this degeneration and the Harder-Narasimhan (HN) filtration of \(N_{E \cup R }|_R\).
In \S \ref{sec:restriction_to_E}, we will prove that \(N_{E \cup R }|_E\) is semistable.
This suffices to prove Theorem \ref{thm:main} when \(g\) is odd by \cite[Lemma 4.1]{clv},
because \(N_{E \cup R }|_R\) is balanced in this case.
When \(g\) is even, \(N_{E \cup R }|_R\) is not balanced. However, we have an explicit geometric understanding of the HN-filtration. In this case, we give two proofs of Theorem~\ref{thm:main}, one using the strong Franchetta Conjecture (see~\S \ref{sec:restriction_to_E}), and an elementary proof using the explicit HN-filtration and induction on \(g\) (in \S \ref{sec:Qsing} and \S \ref{sec:even_g_numerics}). 

\subsection*{Acknowledgments} We would like to thank Atanas Atanasov, Lawrence Ein, Gavril Farkas, Joe Harris, Eric Riedl, Ravi Vakil, and David Yang for invaluable conversations. We also thank the referee for a careful reading of our paper and many valuable suggestions.

\section{Preliminaries}\label{sec-prelim}
In this section, we collect basic facts about normal bundles of reducible curves and their stability. We refer the reader to \cite{clv} for more details.

\subsection{Stability of vector bundles on nodal curves} In our argument, we will specialize canonical curves to certain nodal curves. We recall a natural extension of the notion of stability to nodal curves (see \cite[\S 2]{clv}). Let \(C\) be a connected nodal curve and write \[\nu \colon \tilde{C} \to C\] 
for the normalization of \(C\).  For any node \(p\) of \(C\), let \(\tilde{p}_1\) and \(\tilde{p}_2\) be the two points of \(\tilde{C}\) over \(p\).

Given a vector bundle \(V\) on \(C\), the fibers of the pullback \(\nu^*V\) to \(\tilde{C}\) over \(\tilde{p}_1\) and \(\tilde{p}_2\) are naturally identified.  Given a subbundle \(F \subseteq \nu^*V\), we can thus compare \(F|_{\tilde{p}_1}\) and \(F|_{\tilde{p}_2}\) inside \(\nu^*V|_{\tilde{p}_1} \simeq \nu^*V|_{\tilde{p}_2}\).

\begin{defin}
Let \(V\) be a vector bundle on a connected nodal curve \(C\).  
For a subbundle \(F \subset \nu^*V\), define the \defi{adjusted slope \(\mu^{\text{adj}}_C\)} by 
\[\mu^{\text{adj}}_C(F) \colonequals \mu(F) - \frac{1}{\rk{F}} \sum_{p \in C_{\text{sing}}} \codim_{F} \left(F|_{\tilde{p}_1}\cap F|_{\tilde{p}_2} \right),\]
where \(\codim_F \left(F|_{\tilde{p}_1}\cap F|_{\tilde{p}_2} \right)\) refers to the codimension
of the intersection in either \(F|_{\tilde{p}_1}\) or \(F|_{\tilde{p}_2}\) (which are equal
since \(\dim F|_{\tilde{p}_1} = \dim F|_{\tilde{p}_2}\)). 
Note that if \(F\) is pulled back from \(C\), then \(\mu^{\text{adj}}_C(F) = \mu(F)\).
We say that \(V\) is \defi{(semi)stable} if for all subbundles \(F \subset \nu^*V\),
\[\mu^{\text{adj}}(F) \leqpar \mu(\nu^*V) = \mu(V). \]
\end{defin}

\noindent
The advantage of this definition is that it specializes well.

\begin{prop} \cite[Proposition 2.3]{clv} \label{prop:stab-open}
Let \(\sC \to \Delta\) be a family of connected nodal curves over
the spectrum of a discrete valuation ring,
and \(\sV\) be a vector bundle on \(\sC\).
If the special fiber \(\sV_0 = \sV|_0\) is (semi)stable,
then the general fiber \(\sV^* = \sV|_{\Delta^*}\) is also (semi)stable.
\end{prop}

\begin{lem} \cite[Lemma 4.1]{clv}\label{lem:naive}
Suppose that \(C  = X \cup Y\) is a reducible curve and \(V\) is a vector bundle on \(C\) such that \(V|_X\) and \(V|_Y\) are semistable.
Then \(V\) is semistable.
Furthermore, if one of \(V|_X\) or \(V|_Y\) is stable, then \(V\) is stable.
\end{lem}

\subsection{Elementary modifications of normal bundles}\label{sec-prelim-elementary}

In this section, we briefly recall without proof the definition and basic properties of elementary modifications of a vector bundle on a scheme.  For a more detailed exposition, see \cite[\S2-6]{aly}.  Other treatments can be found in \cite[\S2.3-2.4]{clv} and \cite[\S3.1-3.2]{lv22}.

Given a vector bundle \(V\) on a scheme \(X\), an effective Cartier divisor \(D \subset X\), and a subbundle \(F \subset V|_D\), the \defi{negative elementary modification \(V[D \negmod F]\) of \(V\) along \(D\) towards \(F\)} is defined by the exact sequence
\[0 \to V[D \negmod F] \to V \to V|_D / F \to 0.\]
We write \(V[D\posmod F] \colonequals V[D \negmod F](D)\) for the \defi{positive modification of \(V\) along \(D\) towards \(F\)}.  The modification \(V[D \posmod F]\) is naturally isomorphic to \(V\) on the complement of the divisor \(D\).  In this way we can easily define multiple modifications \(V[D_1 \posmod F_1][D_2 \posmod F_2]\) when the supports of \(D_1\) and \(D_2\) are disjoint.  

When the supports of the \(D_i\) meet, subbundles of \(V|_{D_i}\) are insufficient to define multiple modifications. In this context, we always assume that \(F_i\) extends to a subbundle of \(V\) in an open neighborhood \(U_i\) of \(D_i\).  If \(F_2|_{U_2 \smallsetminus D_1}\) extends to a subbundle over all of \(U_2\) then it does so uniquely and \(V[D_1 \posmod F_1][D_2 \posmod F_2]\) denotes the modification of \(V[D_1 \posmod F_1]\) towards this subbundle along \(D_2\).  
For example, the multiple modification \(V[D \posmod F][D \posmod F]\) denotes the modification of \(V[D \posmod F]\) along \(D\) towards the subbundle of \(V[D \posmod F]\) corresponding to \(F\), which is itself \(F(D)\).
The general situation of multiple modification is studied in \cite[\S2]{aly}.  In this paper when we need multiple modifications, the extension will be clear, and so we won't need this general framework.

\medskip

A simplifying special case of elementary modifications is when \(F\) is a direct summand of \(V\simeq F \oplus V'\) and consequently \(V[D \posmod F] \simeq F(D) \oplus V'\).  
If \(F\) is only a direct summand of \(V|_D \simeq F \oplus G\), we still have an explicit description along \(D\):
\begin{equation}\label{eq:direct_on_D}
E[D \posmod F]|_D \simeq F \otimes \O_D(D) \oplus G.
\end{equation}
\noindent
More generally, if \(V\) sits in an exact sequence
\begin{equation}\label{eq:sample_exact_seq}
0 \to S \to V \to Q \to 0,
\end{equation}
then we obtain an induced exact sequence with the modification \(V[ D \posmod F]\) that captures how the subbundle \(F\) sits with respect to the sequence \eqref{eq:sample_exact_seq}.  We will only make use of the following two special cases of this.  First suppose that \(F \cap S\) is flat over the base \(X\).  In this case \eqref{eq:sample_exact_seq} induces the exact sequence
\begin{equation}\label{eq:exact_flat}
0 \to S[D \posmod (F \cap S)] \to V[D \posmod F] \to Q[D \posmod F/(F\cap S)] \to 0.
\end{equation}
Second, suppose that \(X = C\) is a smooth curve and \(F \subset V\) is a line subbundle.  By combining modifications with disjoint supports, it suffices to consider the case that \(D = np\) for a point \(p \in C\). Let \(k'\) be the order to which the fiber of \(F\) is contained in the fiber of \(S\) in a neighborhood of \(p\). If \(F \) is a subbundle of \(S\), then \(k' = \infty\).  Let \(k = \min(k', n)\).  In this case \eqref{eq:sample_exact_seq} induces the exact sequence
\begin{equation}\label{eq:exact_order}
0 \to S[kp \to F|_{kp}] \to V[np \to F] \to Q[(n-k)p \to \bar{F}] \to 0,
\end{equation}
where \(\bar{F}\) is the saturation of the image of \(F\) in \(Q\).  In the special cases of \(k' = 0\) or \(\infty\) the two sequences \eqref{eq:exact_flat} and \eqref{eq:exact_order} agree.

\medskip

We will primarily work with elementary modifications of the normal bundle of a curve \(C \subset \pp^r\) towards pointing bundles, whose definition we now recall.  Given any linear space \(\Lambda \subset \pp^r\), the projection \(\pi\) from \(\Lambda\), when restricted to \(C\), is unramified on an open \(U_{\Lambda} \subset C\).  If \(U_\Lambda\) is dense in \(C\) and contains \(C^{\text{sing}}\), then the relative tangent sheaf of the map \(\pi\) uniquely extends to a rank \((\dim \Lambda + 1)\) subbundle of \(N_C\), which we denote by \(N_{C \to \Lambda}\) and call the \defi{pointing bundle towards \(\Lambda\)}.  The \defi{pointing bundle exact sequence} is
\[0 \to N_{C \to \Lambda} \to N_C \to \pi^*N_{\pi(C)}(C \cap \Lambda) \to 0.\]
When \(\Lambda \subset \Psi\) are nested subspaces, we have an analogous pointing bundle exact sequence
\begin{equation}\label{eq:pointing2}
    0 \to N_{C \to \Lambda} \to N_{C \to \Psi} \to \pi^*N_{\bar{C} \to \bar{\Psi}}(C\cap \Lambda) \to 0,
\end{equation} 
where \(\bar{\Psi}\) is the projection of \(\Psi\) from \(\Lambda\). 
We abbreviate and write \(N_C[p \posmod \Lambda] \colonequals N_C[p \posmod N_{C \to \Lambda}]\) for modifications towards pointing bundles.

Suppose that \(C\) is a curve on a smooth variety \(X\), and \(M\) is any smooth subvariety
meeting \(C\) quasi-transversely at a point \(p\). Then we write
\[N_C[p \posmodalong M] \colonequals N_C[p \posmod T_p M] \qand N_C[p \negmodalong M] \colonequals N_C[p \negmod T_p M],\]
where \(T_pM\) maps to \(N_C|_p\) via the quotient map \(T_pX \to N_C|_p\).
Observe that when \(M\) is itself a linear space through \(p\), then \(N_C[p \posmodalong M]\) is \emph{not}
isomorphic to \(N_C[p \posmod M]\) because they have different degrees.
Instead, if \(\Lambda \subset M\) is a complementary linear space to \(p\), then
\[N_C[p \posmodalong M] \simeq N_C[p \posmod \Lambda] \qand N_C[p \negmodalong M] \simeq N_C[p \negmod \Lambda].\]
If \(M \cap C = \{p_1, p_2, \ldots, p_n\}\), with all points of intersection
quasi-transverse, then we write
\[N_C[\posmodalong M] \colonequals N_C[p_1 \posmodalong M] \cdots [p_n \posmodalong M].\]

Our interest in modifications towards pointing bundles is rooted in the following result of Hartshorne--Hirschowitz, describing the normal bundle of a nodal curve.

\begin{lem}[{\cite[Corollary~3.2]{HH}}]\label{lem:hh}
Let \(X \cup Y\) be a connected nodal curve in \(\pp^r\). Then
\[N_{X \cup Y}|_X \simeq N_X[\posmodalong Y].\]
\end{lem}

Finally, we recall that the normal bundle of a curve can be related to the normal
bundle of its proper transform in a blowup via modifications.
The simplest case is that of a smooth curve lying on a smooth variety \(C \subset X\),
and a blowup \(\beta \colon \Bl_Y X \to X\) along a smooth subvariety \(Y \subset X\) meeting \(C\)
quasi-transversely at a single point \(p\).
Then the normal bundles of \(C\) in \(X\), and of its proper transform
in the blowup, are related as follows:
\[N_{C / \Bl_Y X} \simeq N_{C/X}[p \negmodalong Y] \quad \text{or equivalently} \quad N_{C/X} \simeq N_{C / \Bl_Y X}[p \posmodalong \beta^{-1}(p)].\]
Via the rules for combining modifications, these formulas immediately imply several generalizations.
We will need the following case:
Suppose that \(Y' \subset Y\) is a smooth subvariety, also passing through \(p\).
Write \(t\) for the natural rational map from the exceptional divisor of \(\Bl_{Y'} X\)
to the exceptional divisor of \(\Bl_Y X\). Then, for any smooth subvariety \(M\) of the image of \(t\)
meeting \(C\) at \(p\):
\begin{equation}\label{eq:blowup_modification}
N_{C / \Bl_Y X}[p \posmodalong M] \simeq N_{C / \Bl_{Y'} X}[p \posmodalong t^{-1}(M)].
\end{equation}

\subsection{The Farey sequence}
Recall that the \defi{\(N\)-Farey sequence} is the sequence of fractions whose denominators are bounded by \(N\) in lowest terms. We refer the reader to \cite{HW} for the properties of the Farey sequence.

\begin{lem}\label{lem:FF}
Let \(V\) be a vector bundle of slope \(\mu(V) = \frac{p}{q}\) in lowest terms and suppose that 
\[0 \to S \to V \to Q \to 0 \]
is an exact sequence of vector bundles such that either \(\mu(S)\) is an adjacent \(q\)-Farey fraction to \(\mu(V)\) with \(\gcd(\deg S, \rk S) = 1\), or similarly for \(Q\).
If both \(S\) and \(Q\) are stable, then any destabilizing subsheaf of \(V\) is isomorphic to either \(S\) or \(Q\).
\end{lem}
\begin{proof}
Suppose that \(V\) has degree \(ep\) and rank \(eq\) for some \(e \geq 1\).
Then the slope of the other bundle (\(\mu(Q)\) or \(\mu(S)\), respectively)
is an adjacent \(eq\)-Farey fraction; this can be seen using the following
two standard properties of adjacent Farey fractions:
\begin{itemize}
\item 
Two rational numbers in lowest terms, \(p_1/q_1\) and \(p_2/q_2\),
are adjacent in the \(\max(q_1, q_2)\)-Farey sequence if and only if
\[\det \begin{bmatrix} p_1 & p_2 \\ q_1 & q_2 \end{bmatrix} = \pm 1.\]
\item 
In this case, they are adjacent in the \(q\)-Farey sequence for any \(\max(q_1, q_2) \leq q < q_1 + q_2\),
and the next fraction appearing between them is 
\[\frac{p_1 + p_2}{q_1 + q_2}.\]
\end{itemize}

There are four cases to consider: \(\mu(S)\) or \(\mu(Q)\) is the next or previous \(eq\)-Farey fraction.
Up to replacing the sequence with its dual,
%(which preserves stability but exchanges subs/quotients and ascending/descending Farey fractions)
it suffices to consider the two cases that \(\mu(S)\) or \(\mu(Q)\) is the next Farey fraction.
Let \(F\) be any subsheaf of \(V\).
Then \(F\) has a filtration
\[0 \to F \cap S \to F \to \Im(F \to Q) \to 0.\]

\smallskip
\noindent
\emph{If \(\mu(S)\) is the next Farey fraction:}
Since \(F \cap S\) is a subsheaf of \(S\), we have
\(\mu(F \cap S) \leq \mu(S)\)
with equality only if \(F\) contains \(S\).  Since \(\mu(V)\) is the previous \(eq\)-Farey fraction to \(\mu(S)\),
if equality does not hold, then \(\mu(F \cap S) \leq \mu(V)\).  Similarly 
\(\mu(\Im(F \to Q)) \leq \mu(Q) < \mu(V)\).
Hence, \(\mu(F) \leq \mu(V)\) unless \(F\) contains \(S\).  Furthermore, if \(F\) properly contains \(S\), then \(\mu(F) < \mu(S)\) and hence \(\mu(F) \leq \mu(V)\) since \(\mu(S)\) is the next Farey fraction.

\smallskip
\noindent
\emph{If \(\mu(Q)\) is the next Farey fraction:}
Similarly,
\(\mu(\Im(F \to Q)) \leq \mu(Q)\)
with equality only if \(F \to Q\) is surjective.
Since \(\mu(V)\) is the previous \(eq\)-Farey fraction to \(\mu(Q)\),
if equality does not hold, then \(\mu(\Im(F \to Q)) \leq \mu(V)\).  Similarly
\(\mu(F \cap S) \leq \mu(S) < \mu(V)\).
Hence, \(\mu(F) \leq \mu(V)\) unless \(F \to Q\) is surjective.  Further, if \(F \to Q\) is not an isomorphism, then \(\mu(F) < \mu(Q)\) and hence \(\mu(F) \leq \mu(V)\) since \(\mu(Q)\) is the next Farey fraction.
\end{proof}

\begin{lem} \label{lem:blackmagic-rationalbase} Suppose that \(\sV\) is a family of vector bundles on a positive-genus curve \(C\)
parameterized by a rational base \(B\).
Suppose that, for \(b_1, b_2 \in B\), the specializations \(\sV|_{b_i}\)
fit into exact sequences
\[0 \to S_i \to \sV|_{b_i} \to Q_i \to 0 \]
satisfying the hypotheses of Lemma~\ref{lem:FF} with \(\mu(S_1) = \mu(S_2)\).
If \(c_1(S_1) \neq c_1(S_2)\), then the general fiber of \(\sV\) is semistable.
\end{lem}
\begin{proof}
Suppose that \(\sV|_b\) is unstable for \(b \in B\) general.  Then there exists a destabilizing subbundle \(\sF \subset \sV|_b\).
Consider the rational map
\[c_1(\sF) \colon B \dashrightarrow \Pic C.\]
Since \(B\) is rational, this map is constant.

On the other hand, we may specialize to the fiber over \(b_i\).
As we approach along any arc, \(\sF|_b\) limits to one of \(S_i\) or \(Q_i\) (based on which one has slope greater than \(\mu(\sV)\)) by Lemma~\ref{lem:FF}. Therefore, \(c_1(\sF)\) extends to a regular map in a neighborhood
of \(b_i\). Our assumption
that \(c_1(S_1) \neq c_1(S_2)\) (and so also \(c_1(Q_1) \neq c_1(Q_2)\))
then gives a contradiction.
\end{proof}

\subsection{\boldmath Natural bundles on a genus \(1\) curve} Let \(E\) be a genus \(1\) curve.  We say that a map \(f \colon \Pic^a E \to \Pic^bE\) is \defi{natural} if for any automorphism \(\theta \colon E \to E \), the following diagram commutes:

\begin{center}
\begin{tikzcd}
\Pic^a E \arrow[r, "f"] \arrow[d, "\theta^*"]&\Pic^b E \arrow[d, "\theta^*"] \\
\Pic^a E \arrow[r, "f"]&\Pic^b E
\end{tikzcd}
\end{center}

\begin{lem}\label{lem:blackmagic-natural}
If \(f \colon \Pic^a E \to \Pic^b E\) is natural, then \(a\) divides \(b\).
\end{lem}
\begin{proof}
Translation by a point of order \(a\) is the identity on \(\Pic^aE\), and so must also be on \(\Pic^bE\).
\end{proof}

\section{Our degeneration}\label{sec:our_degen}

Let \(E \subset \pp^{g-1}\) be an elliptic normal curve.  Let \(H\simeq \pp^{g-2}\) be a general hyperplane and let \(\Gamma \colonequals E \cap H\) be the hyperplane section of \(E\).  Let \(R\) be a general rational curve of degree \(g-2\) in \(H\), meeting \(E\) quasi-transversely at the points of \(\Gamma\).  Then by  \cite[Lemma 5.7]{lv22}, the curve \(E \cup R\) is a Brill--Noether curve of degree \(2g - 2\) and genus \(g\); i.e., it is a degeneration of a canonical curve.

\begin{lem}[{\cite[Lemma 5.8 and Proposition 13.7]{lv22}}]\label{lem:N_restrict_R}
We have
\[N_{E \cup R}|_R \simeq \begin{cases} \O(g+1)^{\oplus (g-2)} &: \text{\(g\) odd} \\ \O(g) \oplus \O(g+1)^{\oplus (g-4)} \oplus \O(g+2) &: \text{\(g\) even.} \end{cases}\]
\end{lem}

By \cite[Lemma 4.1]{clv}, when \(g\) is odd, it suffices to show that \(N_{E \cup R}|_E\) is semistable to conclude that the normal bundle of a general canonical curve is semistable.  This is addressed in Section \ref{sec:restriction_to_E}.
When \(g\) is even, we will need to know that \(N_{E \cup R}|_E\) is semistable, and also that certain modifications of \(N_{E \cup R}|_E\), related to the Harder--Narasimhan (HN) filtration of \(N_{E \cup R}|_R\), are semistable.  We conclude this section with a brief geometric description of the HN-filtration, expanding on \cite[Section 13]{lv22}.

\subsection{\boldmath The HN-filtration when \(g\) is even}\label{sec:HN}

In this section, we suppose that \(g = 2n+2\) is even.  We first recall without proof some results we will need from \cite[Section 13]{lv22}.
Suppose that \(E \subset \pp^{2n+1}\) is an elliptic normal curve.  Let
\(p_1+ \cdots + p_{2n+2}\) be a general section of \(O_E(1)\).  Let \(q_1, \dots, q_{2n+2}\) be general points on \(\pp^1\).  By \cite[Lemma 13.1]{lv22}, there are exactly two degree \(n+1\) maps
\[f_i \colon E \to \pp^1\]
sending \(p_j\) to \(q_j\) for all \(1 \leq j \leq 2n+2\) (see also \cite{cps} and \cite{fli} for more general results of the type).  Together, these define a map 
\[\bar{f} \colon E \to \pp^1 \times \pp^1,\]
which is birational onto an \((n^2-1)\)-nodal curve of bidegree \((n+1, n+1)\) \cite[Lemma 13.2]{lv22}, none of whose nodes lie on the diagonal.

Let \(S\) denote the blowup of \(\pp^1 \times \pp^1\) at the \(n^2 - 1\) nodes of \(\bar{f}(E)\), with total exceptional divisor \(F\), and write
\(f \colon E \hookrightarrow S\)
for the resulting embedding.  By \cite[Lemma 13.3]{lv22}, the line bundle \(L = \O_S(n,n)(-F) = K_S(1,1)(E)\) on \(S\) restricts to
\[ L|_E \simeq \O_E(1,1) \simeq \O_E(p_1 + \cdots + p_{2n+2}).\]
Write \(\pi_i \colon S \to \pp^1\) for the two projections onto each factor of \(\pp^1 \times \pp^1\).  As computed in \cite{lv22}, \({\pi_i}_*L(-E) = R^1{\pi_i}_*L(-E) = 0\), and so \(H^0(L(-E)) = H^1(L(-E)) = 0\) (by the Leray spectral sequence), and so \(H^0(L) \simeq H^0(L|_E)\).  Furthermore,
by \cite[Equation~(196)]{lv22},
\[(\pi_i)_*L \simeq (f_i)_*\O_E(p_1 + \cdots p_{2n+2}) \simeq \O_{\pp^1}(1)^{\oplus(n+1)}.\]
The map \(S \to \pp H^0(L) \simeq \pp H^0(L|_E)  \simeq \pp^{2n+1}\) given by \(|L|\) thus factors through the balanced scrolls 
\[\Sigma_i = \pp [(\pi_i)_*L]  \simeq \pp^1 \times \pp^n\]
embedded by the relative \(\O(1)\), and is hence an embedding. 
Let \(R\) denote the diagonal of \(\pp^1\times \pp^1\), viewed as a divisor on \(S\).
By construction, \(R\) meets \(E\) at \(p_1, \dots, p_{2n+2}\).
Along \(R\), the bundle \(L|_R\) has degree \(2n\), and hence maps \(R\) into a hyperplane in \(\pp^{2n+1}\). The reducible curve \(E \cup R\) is a degeneration of a canonical curve.

Finally, we recall a construction of Zamora \cite[Lemma 1.1]{zamora} of a rank \(4\) quadric in \(\pp^{2n+1}\) containing \(E\), and show that it also contains the scrolls \(\Sigma_1\) and \(\Sigma_2\).
Let \(s_1, s_2\) be a basis for the linear system giving rise to the first map \(f_1 \colon E \to \pp^1\) and let \(t_1, t_2\) be a basis for the linear system giving rise to the second map \(f_2 \colon E \to \pp^1\).  Then the \(s_i \otimes t_j\) are sections of \(\O(1,1)|_E = L|_E\), and we may therefore view them as linear functions on the \(\pp^{2n+1}\).  Furthermore, as a section of \(L|_E^{\otimes 2}\), 
\[\det \begin{pmatrix} s_1\otimes t_1 & s_2 \otimes t_1 \\ s_1 \otimes t_2 & s_2 \otimes t_2 \end{pmatrix} = 0\]
This determinant defines a rank \(4\) quadric \(Q \subset \pp^{2n+1}\) containing \(E\).  Changing the bases \(s_1, s_2\) or \(t_1, t_2\) corresponds to a row/column operation, so this quadric is independent of the choice of basis.

To see that the quadric contains \(\Sigma_1\), we will show that it contains every fiber \(\pp^n = \operatorname{Span} (f_1^{-1}(x))\) for \(x \in \pp^1\).  Choose a basis so that the first element \(s_1\) vanishes on \(f_1^{-1}(x)\).  Thus the linear functions corresponding to \(s_1 \otimes t_1\) and \(s_1 \otimes t_2\) vanish along \(\operatorname{Span} (f_1^{-1}(x))\) in \(\pp^{2n+1}\), and hence the quadric \(Q\) contains this plane.  Varying \(x\), we see that \(Q\) contains \(\Sigma_1\).  Similarly \(Q\) contains \(\Sigma_2\).
Putting all of this together, we can summarize this situation with the following setup:

\begin{setup}\label{eq:S_sigma_Q_diagram}
Given an elliptic curve \(E \subset \pp^{2n+1}\) and two maps \(f_i \colon E \to \pp^1\), we obtain the following inclusions:
\begin{center}
\begin{tikzcd}
& & \Sigma_1 \arrow[rd] \\
E \cup R \arrow[r] & S \arrow[ru] \arrow[rd] & & Q \\
& & \Sigma_2 \arrow[ru]
\end{tikzcd}
\end{center}
Moreover, projection from \(Q_\sing\) induces maps \(\Sigma_i \to \pp^1 \times \pp^1 \subset \pp^3\), whose composition with projection onto the \(i\)th \(\pp^1\) factor is the structure map for the projective bundle. In particular, the composition of the inclusion \(S \hookrightarrow Q\) with projection from \(Q^\sing\) is identified with the blowup map \(S \to \pp^1 \times \pp^1\).
\end{setup}
\noindent
The maps in Setup \ref{eq:S_sigma_Q_diagram} give rise to a filtration of \(N_{E \cup R}\),
\begin{equation}\label{eq:HN_R}
N_{E \cup R / S} \subset N_{E \cup R / Q} \subset N_{E \cup R}.
\end{equation}

\begin{prop}[{\cite[Proposition 13.7]{lv22}}]
The restriction of \eqref{eq:HN_R} to \(R\) is the HN-filtration of \(N_{E \cup R}|_R\).
\end{prop}
\begin{rem}
In \cite[Proposition 13.7]{lv22}, the middle piece of the filtration does not have a geometric description.
Instead, it is described as ``\(N_{E \cup R / \Sigma_1} + N_{E \cup R/ \Sigma_2}\)" --- which is equal to \(N_{E \cup R / Q}\) since it is contained in it, and has the same rank and degree.
\end{rem}

\begin{prop}
Fixing a line bundle \(\O_E(1)\) of degree \(2n+2\) on an elliptic curve \(E\), the set of possible \(S, R, \Sigma_1, \Sigma_2, Q\)  in Setup \ref{eq:S_sigma_Q_diagram} varies in a rational base.
\end{prop}
\begin{proof} 
The data in \eqref{eq:S_sigma_Q_diagram} is determined by the following choices: 
\begin{enumerate}
\item A basis (up to common scaling) for \(H^0(\O_E(1))\), which determines the embedding \(E \subset \pp^{2n+2}\). The choice of a basis of a vector space depends on a rational base.
\item An unordered pair of line bundles \(f_i^* \O_{\pp^1}(1)\) that sum to \(\O_E(1)\). The space of line bundles of a fixed degree on \(E\) can be identified with \(E\). This choice corresponds to the fiber of the map \(a \colon (E \times E)/ \mathfrak{S}_2 \to E\) given by addition over \(\O_E(1)\).  The surface \((E \times E)/ \mathfrak{S}_2\) is a ruled surface over \(E\), so this choice is rational.
\item Two sections (up to common scaling) of each of these line bundles (defining 
\(f_i \colon E \to \pp^1\)). As in (1), this choice depends on a rational base.
\end{enumerate}
We conclude that the set of possible \(S, R, \Sigma_1, \Sigma_2, Q\)  in Setup \ref{eq:S_sigma_Q_diagram} varies in a rational base.
\end{proof}

\section{Semistability of the restriction to \(E\)}\label{sec:restriction_to_E}

In this section, we show that the restricted normal bundle \(N_{E\cup R}|_E\), where \(E \cup R \subset \pp^{g-1}\) is the degenerate canonical curve  introduced in Section~\ref{sec:our_degen}, is semistable.  

\begin{thm}\label{thm:restriction_ss}
If \(g \not\in \{4, 6\}\), then \(N_{E \cup R}|_E\) is semistable.
\end{thm}

We will first show
that \(N_{E \cup R}|_E\) is ``close-enough-to-semistable'' that no naturally defined destabilizing subbundles could exist.
We have that
\[\mu(N_{E \cup R}|_E) = \frac{g(g+1)}{g-2} = g + 3 + \frac{6}{g-2}.\]
The fractional part of the slope depends on \(g\) modulo \(6\).  Write
\[g-2 = 6k + \epsilon \qwhere 0 \leq \epsilon < 6.\]

\begin{lem}\label{prop:N_restrict_E_close_ss}
The bundle \(N_{E \cup R}|_E\) has no subbundles of slope greater than \(g+3 + \frac{1}{k}\) and no quotient bundles of slope less than \(g+3 + \frac{1}{k+1}\).
\end{lem}

\noindent We will deduce this from taking \(m = 0\) in
the following more general statement.

\begin{lem}\label{lem:restriction_all_but}
Suppose that \(E \subseteq \pp^{g-1}\) is an elliptic normal curve and \(R\) is a general \(g\)-secant rational curve of degree \(g-2\).  Let \(0 \leq m \leq 5\).  Write \(g-2 = (6-m)k + \epsilon\) for \(0 \leq \epsilon  < 6-m\).
Then 
\[N_E[p_1 + \cdots + p_{g - m}\posmodalong R] \]
has no subbundles of slope greater than \(g+3 + \frac{1}{k}\), and no quotient bundles of slope less than \(g+3 + \frac{1}{k+1}\).
\end{lem}

\noindent In the course of proving Lemma \ref{lem:restriction_all_but}, we will need the following result.

\begin{lem}\label{lem:pointing_ss}
Let \(n \geq 2\) be an integer and 
suppose that \(\Lambda \subset \pp^{g-1}\) is a quasi-transverse \(n\)-secant \((n-1)\)-plane to \(E\).  Let \(q \in \Lambda\) be a general point.  Suppose that \(R \subset \Lambda\) is a general rational curve of degree \(n-1\) through \(E \cap \Lambda\) and \(q\).
Let \(y\) be a general point on \(E\).  Then the modified pointing bundle
\[N_{E \to \Lambda}[\posmodalong R][y \posmod q]\] 
is stable of slope \(g+3 + \frac{1}{n}\).
\end{lem}
\begin{proof}
We will prove this by induction on \(n\).  Specialize \(q\) to one of the points \(p\) where \(R\) meets \(E\).  If \(n > 2\), then the pointing bundle exact sequence \eqref{eq:pointing2} towards \(p\) induces the sequence
\[0 \to N_{E \to p}(y) \to N_{E \to \Lambda}[\posmodalong R][y \posmod p] \to N_{\bar{E} \to \bar{\Lambda}}[\posmodalong \bar{R}](p) \to 0, \]
as in \eqref{eq:exact_flat} and \eqref{eq:exact_order}.
The subbundle \(N_{E \to p}(y)\) is isomorphic to \(\O_E(1)(2p + y)\), which is stable of slope \(g+3\).  The quotient is a twist of another instance of our problem in \(\pp^{g-2}\).  We may therefore assume by induction that it is stable of slope \(g+3 + \frac{1}{n-1}\).  
Since \(c_1(\O_E(1)(2p + y))\) depends on the choice of the point \(p\), we conclude by Lemma  \ref{lem:blackmagic-rationalbase} that the general fiber is semistable (and hence stable) as desired.

It suffices, therefore, to treat the base case of \(n=2\).  In this case, \(R=\Lambda\) is a \(2\)-secant line \(\bar{pp'}\), and after specializing as above, the pointing bundle exact sequence towards \(p\) is
 \[0 \to N_{E \to p}(y + p') \to N_{E \to \Lambda}[p \posmod p'][p' \posmod p][y \posmod p] \to N_{\bar{E} \to p'}(2p) \to 0. \]
In this case the subbundle and quotient bundles are stable line bundles of slopes \(g+4\) and \(g+3\) respectively.  Again, applying Lemma  \ref{lem:blackmagic-rationalbase}, we conclude that the general fiber is semistable (and hence stable) as desired.
\end{proof}

\begin{proof}[Proof of Lemma \ref{lem:restriction_all_but}]
Our argument will be by backwards induction on \(m\).  The base case of \(m=5\) is Lemma \ref{lem:all_but_5} below, so we suppose \(m \leq 4\).

We first prove the upper bound on the slope of a subbundle by exhibiting a degeneration that lies in an exact sequence with a subbundle that is stable of slope exactly \(g+3 + \frac{1}{k}\) and quotient which satisfies our inductive hypothesis.
Let \(\Lambda_1 \simeq \pp^{k-1} \subset \pp^{g-1}\) be the span of the first \(k\) points \(p_1, \dots, p_k\) of \(E \cap R\).  Let \(\Lambda_2 \simeq \pp^{g-k-2}\) be the span of the last \(g-1-k\) points \(p_{k+2}, \dots, p_{g}\).  Since the remaining point \(p_{k+1}\) is constrained to lie in the hyperplane spanned by the other points, there is a unique line \(L\) through \(p_{k+1}\) that meets both \(\Lambda_1\) and \(\Lambda_2\).  

Let \(x_1\) and \(x_2\) denote the points where \(L\) meets \(\Lambda_1\) and \(\Lambda_2\), respectively.
Let \(R_1\) be a general rational curve in \(\Lambda_1\) of degree \(k-1\) through \(p_1, \dots, p_k, x_1\), and let \(R_2\) be a general rational curve in \(\Lambda_2\) of degree \(g-k-2\) through \(p_{k+2}, \dots, p_{g}, x_2\).   Then 
\[R^\circ \colonequals L \cup R_1 \cup R_2\] 
is a degeneration of \(R\).  It suffices to prove that \(N_E[p_1 + \cdots + p_{g - m}\posmodalong R^\circ]\) has no subbundles of slope greater than \(g+3 + \frac{1}{k}\) to prove the lemma.
Consider the pointing bundle exact sequence for pointing towards the subspace \(\Lambda_1\):
\[0 \to N_{E \to \Lambda_1}[\posmodalong R_1][p_{k+1} \posmod x_1] \to N_E[p_1 + \cdots  + p_{g - m}\posmodalong R^\circ] \to N_{\bar{E}}[p_{k+2} + \dots + p_{g - m}\posmodalong \bar{R_2}](p_1 + \cdots + p_k) \to 0.\]

In order to use Lemma \ref{lem:pointing_ss} to show that
\(N_{E \to \Lambda_1}[\posmodalong R_1][p_{k+1} \posmod x_1]\) is stable of slope \(g+3 + \frac{1}{k}\), we need that, as the points \(p_{k+1}, \dots, p_g\) vary, the point \(x_1\) is general in \(\Lambda_1\).  That is, there are no obstructions to lifting a deformation of the point \(x_1\) to a deformation of the plane \(\Lambda_2' \colonequals \bar{\Lambda_2, x_1}\) (maintaining the necessary incidences with \(E\)).   
These obstructions live in \(H^1(\Lambda_2', N)\), where the bundle \(N\) is the kernel of the map
\[N_{\Lambda_2'} \to \left. N_{\Lambda_2'} \right|_{x_1} \oplus \bigoplus_{i=k+1}^g \left. N_{\bar{\Lambda_2', T_{p_i}E}}\right|_{p_i}. \]

The key numerical input is \(2k \leq g\), which follows from \(m \leq 4\).
Since \(\Lambda_2'\) is the complete intersection of the \(k\) hyperplanes spanned by \(\Lambda_2'\) and all but one of the tangent lines \(T_{p_i}E\) for \(k+1\leq i \leq 2k \leq g\),
the bundle \(N\) sits in the exact sequence
\[0 \to \bigoplus_{i=k+1}^{2k} N_{\Lambda_2'/\bar{\Lambda_2', T_{p_i}E}} \otimes \sI_{x_1 \cup p_{k+1} \cup \dots \cup \hat{p}_i \cup \dots \cup p_g} \to N \to P\to 0,\]
where \(P\) is a punctual sheaf (and hence \(h^1(P) = 0\)).  Moreover, the evaluation map \(\ev\) in
\[0 \to N_{\Lambda_2'/\bar{\Lambda_2', T_{p_i}E}} \otimes \sI_{x_1 \cup p_{k+1} \cup \dots \cup \hat{p}_i \cup \dots \cup p_g} \to \O_{\Lambda_2'}(1) \xrightarrow{\ev} \O_{\Lambda_2'}(1)|_{x_1 \cup p_{k+1} \cup \dots \cup \hat{p}_i \cup \dots \cup p_g} \to 0\]
is surjective on global sections, since the points \(x_1 \cup p_{k+1} \cup \dots \cup \hat{p}_i \cup \dots \cup p_g\) form a basis for the plane \(\Lambda_2'\),
and \(h^1(\O_{\Lambda_2'}(1)) = 0\). Therefore
\[H^1\left(N_{\Lambda_2'/\bar{\Lambda_2', T_{p_i}E}} \otimes \sI_{x_1 \cup p_{k+1} \cup \dots \cup \hat{p_i} \cup \dots \cup p_g}\right) = 0,\]
and hence \(H^1(N) = 0\).

In the quotient, \(N_{\bar{E}}[p_{k+2} + \dots + p_{g - m}\posmodalong \bar{R_2}]\) is another case of our inductive hypothesis with one fewer modification occurring at the points of incidence of \(\bar{R_2}\) with \(\bar{E}\) (with a larger value of $k$ if $\epsilon = 5-m$).
The result now follows from our inductive hypothesis.

Now we turn to the lower bound on the slope of any quotient.  We will exhibit a specialization that lies in an exact sequence with a subbundle that is stable of slope exactly \(g+3 + \frac{1}{k+1}\) and a quotient bundle 
which satisfies our inductive hypothesis.  We will modify the same argument by letting \(\Lambda_1\) be the \(k\)-dimensional span of \(p_1, \dots, p_{k+1}\) and letting \(\Lambda_2\) be a the \((g-k-3)\)-dimensional span of  \(p_{k+3}, \dots, p_{g}\).  As above, there is a unique line \(L\) through the remaining point \(p_{k+2}\) that meets both \(\Lambda_1\) (at a point \(x_1\)) and \(\Lambda_2\) (at a point \(x_2\)).  
We define \(R_1\) and \(R_2\) analogously to above.
In the pointing bundle exact sequence towards \(\Lambda_1\): \\
\scalebox{0.95}{\begin{minipage}{1.05\textwidth}
\[0 \to N_{E \to \Lambda_1}[\posmodalong R_1][p_{k+2} \posmod x_1] \to N_E[p_1 + \cdots  + p_{g - m}\posmodalong R^\circ] \to N_{\bar{E}}[p_{k+3} + \dots + p_{g - m}\posmodalong \bar{R_2}](p_1 + \cdots + p_{k+1}) \to 0,\]
\end{minipage}} \\[0.5ex]
the subbundle is stable of slope \(g+3 + \frac{1}{k+1}\) by Lemma \ref{lem:pointing_ss} (using the same argument to ensure generality of \(x_1\)), and the quotient is a twist of another case of our inductive hypothesis in \(\pp^{g-k-2}\) (with a smaller value of \(k\) if \(\epsilon = 0\)), with one fewer modification occurring along \(\bar{R_2}\).

This completes the inductive step.
All that remains is therefore to verify the base case,
which is Lemma~\ref{lem:all_but_5} below.
\end{proof}

\begin{lem}\label{lem:all_but_5}
Suppose that \(E \subset \pp^{g-1}\) is an elliptic normal curve, and \(R\) is a degree \(g - 2\) rational curve meeting \(E\) at \(p_1, \dots, p_g\) quasi-transversely.  Then
\[N_E' \colonequals N_E[ p_1 + \cdots + p_{g-5} \posmodalong R]\]
is stable of slope \(g + 3 + \frac{1}{g-2}\).
\end{lem}
\begin{proof}
We will prove this by induction on \(g\).  The base case is \(g=5\), in which case \(N_E' = N_E\) is stable by \cite{einlazarsfeld}.  
Otherwise, when \(g \geq 6\), the bundle \(N_E'\) is modified at \(p_1\). Let \(\Lambda \simeq \pp^{g-3}\) be the span of \(p_2, \dots, p_{g-1}\).  Let \(L\) be the line through \(p_1\) and \(p_g\) that meets \(\Lambda\) at a point \(x\).  Let \(R'\) be a rational curve of degree \(g-3\) through \(p_2, \dots, p_{g-1}, x\).  Then \(R^\circ = R' \cup L\) is a degeneration of \(R\).  Consider the specialization
\[N_E[ p_1 + \cdots + p_{g-5} \posmodalong R^\circ]\]
of \(N_E'\).
Consider the pointing bundle exact sequence for pointing towards \(p_g\):
\[0 \to N_{E \to p_g}(p_1) \to N_E[ p_1 + \cdots + p_{g-5} \posmodalong R^\circ] \to N_{\bar{E}}(p_g)[ p_2 + \cdots + p_{g-5} \posmodalong \bar{R'}]\to 0.\]
The subbundle has slope \(g+3\) exactly.  Since \(\bar{R'}\) is a rational curve of degree \(g-3\) meeting \(\bar{E}\) at 
\(p_2, \dots,  p_{g-1}, x\),
the quotient bundle is a twist of an instance the same problem in \(\pp^{g-2}\).
By induction it is stable.
Moreover, \(c_1(N_{E \to p_g}(p_1))\) depends on the ordering of \(p_1, p_2, \ldots, p_g\).
Hence by Lemma \ref{lem:blackmagic-rationalbase},
the general fiber \(N_E'\) is semistable (thus stable) as desired.
\end{proof}

We complete the proof by appealing to the naturality of the maximal destabilizing subbundle, and using the following purely combinatorial lemma.

\begin{lem}\label{lem:numerical}
Let \(k \geq 0\) and \(0 \leq \epsilon < 6\) be integers with
\[(k, \epsilon) \notin \{(0,2), (0,4)\}.\] 
Then there are no integers \(r, d\) satisfying
\begin{gather}
1 \leq r < 6k+\epsilon, \text{ and} \label{cond1} \\
6k + 5 + \epsilon  + \frac{6}{6k+\epsilon} < \frac{d}{ r} \leq 6k+ 5 + \epsilon  + \frac{1}{k} , \text{ and} \label{cond2}\\
6k+ 5 + \epsilon  + \frac{1}{k+1} \leq \frac{(6k+5 + \epsilon)(6k + \epsilon) + 6 - d}{6k + \epsilon - r}, \text{ and} \label{cond3}\\
(6k + 2 + \epsilon) \mid d. \label{cond4}
\end{gather}
\end{lem}
\begin{proof}
Suppose such integers \(d\) and \(r\) exist.
Clearing denominators, \eqref{cond2} and \eqref{cond3} yield:
\begin{align*}
(6k + \epsilon) d - (6k + \epsilon + 2)(6k + \epsilon + 3) r &> 0 \\
-k d + (6 k^2 + k \epsilon + 5k + 1) r &\geq 0 \\
-(k + 1) d + (6k^2 + k \epsilon + 11k + \epsilon + 6) r &\geq \epsilon - 6.
\end{align*}
Adding \(6 - \epsilon\) times the second of these inequalities to \(\epsilon\) times the third yields
\[(6k + \epsilon) d - (6k + \epsilon + 2)(6k + \epsilon + 3) r \leq \epsilon (6 - \epsilon).\]
Combined with the first, we learn that the integer
\[X \colonequals (6k + \epsilon) d - (6k + \epsilon + 2)(6k + \epsilon + 3) r\]
satisfies
\[0 < X \leq \epsilon (6 - \epsilon).\]
On the other hand, by \eqref{cond4},
\[X = (6k + \epsilon) d - (6k + \epsilon + 2)(6k + \epsilon + 3)r \equiv 0 \mod 6k + \epsilon + 2.\]

It follows that \(\epsilon (6 - \epsilon) \geq 6k + \epsilon + 2\),
or upon rearrangement, \(6k + 2 \leq \epsilon (5 - \epsilon)\).
Since \(\epsilon\) is an integer with \(0 \leq \epsilon \leq 5\), we have
\(\epsilon (5 - \epsilon) \leq 6\), and so \(6k + 2 \leq 6\),
which implies \(k = 0\).
Moreover, if \(\epsilon = 0\) or \(\epsilon = 5\), then \(\epsilon (5 - \epsilon) = 0\),
in violation of \(6k + 2 \leq \epsilon (5 - \epsilon)\).
The cases \((k, \epsilon) = (0, 2)\) and \((0, 4)\) are excluded by assumption,
so the only remaining cases are \((k, \epsilon) = (0, 1)\) and \((0, 3)\):
\begin{itemize}
\item When \((k, \epsilon) = (0, 1)\), we have \(1 \leq r < 6k+\epsilon = 1\), which is a contradiction.
\item When \((k, \epsilon) = (0, 3)\), we have \(0 < X \leq 9\) and \(X \equiv 0\) mod \(5\). Therefore \(3d - 30 r = X = 5\),
which is a contradiction by looking mod \(3\).
\end{itemize}
This completes the proof.
\end{proof}

\subsection*{Proof of Theorem \ref{thm:restriction_ss}}

Let \((d, r)\) be the degree and rank of the maximal destabilizing subbundle of \(N_{E \cup R}|_E\).  Since this naturally-defined bundle depends only on the choice of \(\O_E(1)\) plus choices varying in a rational base, its determinant gives a natural map \(\Pic^g E \to \Pic^d E\).
By Lemma \ref{lem:blackmagic-natural}, the degree \(d\) is divisible by \(g\).
By Lemma \ref{prop:N_restrict_E_close_ss}, the slope \(d/r\) is at most \(g+3 + \frac{1}{k}\), with quotient bundle having slope at least \(g+3 + \frac{1}{k+1}\).    By Lemma \ref{lem:numerical}, no such integers \(d\) and \(r\) exist, and hence no destabilizing bundles exist, when \(g \not\in \{4,6\}\). \qed

\subsection*{Proof of Theorem \ref{thm:main} in odd genus}

Let \(C\) be a general canonical curve of odd genus \(g \geq 3\).
By \cite[Lemma 4.1]{clv}, semistability of \(N_C\) follows from the semistability of \(N_{E \cup R}|_E\) and \(N_{E \cup R}|_R\).  The first of these is Theorem \ref{thm:restriction_ss}; the second is Lemma \ref{lem:N_restrict_R}. \qed

\subsection*{Proof of Theorem \ref{thm:main} in even genus using the Strong Franchetta Conjecture} The proof of Theorem \ref{thm:main} in even genus is considerably harder. Here we will give an argument using the Strong Franchetta Conjecture proved by  Harer \cite{harer} and Arbarello and Cornalba \cite{arbarellocornalba, arbarellocornalba2} in characteristic 0 and Schr\"{o}er \cite{schroer} in characteristic \(p\). In next two sections, we will give an elementary proof.

Suppose that the normal bundle of the general canonical curve is unstable. Specialize to \(E \cup R\) as in Section~\ref{sec:our_degen}. If \(g \geq 8\), then \(N_{E \cup R}|_E\) is semistable by Theorem \ref{thm:restriction_ss}, and any destabilizing subbundle of \(N_{E\cup R}|_R\) of rank \(r\) has slope at most \(\mu(N_{E\cup R}|_R) + \frac{1}{r}\) by Lemma \ref{lem:N_restrict_R}.  Consequently, if \(g \geq 8\), then the maximal destabilizing subbundle \(F\) of \(N_C\) would
satisfy
\[\mu(N_C) < \mu(F) \leq \mu(N_C) + \frac{1}{r} \qwhere r = \rk F.\]
On the other hand, by the Strong Franchetta Conjecture, \(\det F\)
is a multiple of the canonical bundle.  We conclude that the degree of \(F\) is \(s(2g-2)\) for some integer \(s\). 
Since the slope of the normal bundle of a canonical curve is \((g+1)(2g-2)/(g - 2)\), we obtain the inequality
\[\frac{(g+1)(2g-2)}{g-2} < \frac{s(2g-2)}{r} \leq \frac{(g+1)(2g-2)}{g-2} + \frac{1}{r},\]
or upon rearrangement,
\[0 < (s-r)(g-2)-3r \leq \frac{g-2}{2g-2} < 1.\]
Since \((s-r)(g-2) -3r\) is an integer, this is a contradiction.
Hence, \(N_C\) is semistable for the general canonical curve.
\qed

\section{Degeneration so that \(Q_{\text{sing}}\) meets \(E\)}\label{sec:Qsing}
In order to give an elementary proof of Theorem \ref{thm:main} in the even genus case using the explicit description of the HN-filtration given in Section \ref{sec:HN}, we will show in Section~\ref{sec:even_g_numerics} that it suffices to bound the slopes of subbundles of \(N_{E/Q}[2\Gamma   \posmod N_{E/S}]\).
To achieve such a bound, we will utilize a further degeneration in which \(E\) meets the singular locus \(Q_{\text{sing}}\) of the rank \(4\) quadric \(Q\) described in Section \ref{sec:HN} in two points \(\{x_1, x_2\}\).
The basic inductive strategy will be to degenerate in this way, and then examine
the sequence obtained by projection from the line \(\overline{x_1 x_2}\).
If we do this carefully, the quotient will be another instance
of our Setup \ref{eq:S_sigma_Q_diagram} in \(\pp^{2n - 1}\).  In this section, we construct this degeneration and prove that the projection exact sequence behaves as desired.  In the next section, we will use this to complete our inductive proof of Theorem \ref{thm:main} in the even genus case.

We will construct this degeneration from an instance
\((\bar{E}, \bar{R}, \bar{S}, \bar{Q})\) of Setup \ref{eq:S_sigma_Q_diagram} in
\(\pp^{2n - 1}\). 
The basic strategy will be to construct a degenerate instance of Setup \ref{eq:S_sigma_Q_diagram} by specializing the smooth elliptic curve of type \((n+1, n+1)\) on (the blowup of) \(\pp^1 \times \pp^1\) to the union of a smooth elliptic curve of type \((n,n)\) union a \((1,0)\) curve and a \((0,1)\) curve.
Write \(\bar{\Gamma} = \bar{E} \cap \bar{R}\).  Recall that via the given maps \(\bar{f}_1\) and \(\bar{f}_2\), \(\bar{E}\) maps to \(\pp^1 \times \pp^1\); \(\bar{R}\) corresponds to the diagonal in \(\pp^1 \times \pp^1\).  We illustrate this below.
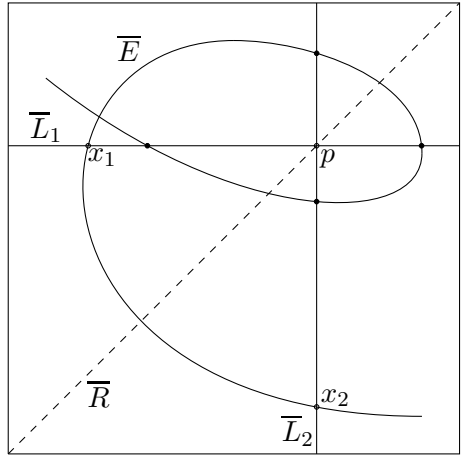
\begin{center}
\begin{tikzpicture}
\draw (0, 0) -- (6, 0) -- (6, 6) -- (0, 6) -- (0, 0);
\draw[dashed] (0, 0) -- (6, 6);
\draw (0.5, 5) .. controls (3, 3)  and (5.5, 3) .. (5.5, 4);
\draw (5.5, 4) .. controls (5.5, 5) and (4, 5.5) .. (3, 5.5);
\draw (3, 5.5) .. controls (0, 5.5) and (0, 0.5) .. (5.5, 0.5);
\draw (4.1, 0) -- (4.1, 6);
\draw (0, 4.1) -- (6, 4.1);
\draw (1.25, 3.95) node{\(x_1\)};
\draw (4.35, 0.75) node{\(x_2\)};
\draw (4.25, 3.93) node{\(p\)};
\draw (0.50315, 4.35) node{\(\bar{L}_1\)};
\draw (3.85, 0.3125) node{\(\bar{L}_2\)};
\draw (1.063, 4.1) circle[radius=0.03];
\draw (4.1, 4.1) circle[radius=0.03];
\filldraw (1.85, 4.1) circle[radius=0.03];
\filldraw (5.495, 4.1) circle[radius=0.03];
\filldraw (4.1, 5.33) circle[radius=0.03];
\filldraw (4.1, 3.357) circle[radius=0.03];
\draw (4.1, 0.625) circle[radius=0.03];
\filldraw (1.2, 0.8) node{\(\bar{R}\)};
\filldraw (1.6, 5.4) node{\(\bar{E}\)};
\end{tikzpicture}
\end{center}

We take \(x_1, x_2 \in \bar{E}\) so that \(\bar{f}_1(x_1) = \bar{f}_2(x_2)\),
and write \(p = (\bar{f}_1(x_1), \bar{f}_2(x_2))\).
Let \(\bar{L}_1 = \bar{f}_1(x_1) \times \pp^1\) and \(\bar{L}_2 = \pp^1 \times \bar{f}_2(x_2)\) denote the
corresponding lines of the ruling (which meet at \(p\)).  
Let \(\Delta\) denote the remaining set (not including
\(\{x_1, x_2\}\)) of points where one of the \(\bar{L}_i\) meets \(\bar{E}\),
together with \(p\) and the nodes of \(\bar{E}\).
Construct the blowup \(S^\circ\) of \(\pp^1 \times \pp^1\) at \(\Delta\),
and write
\begin{align*}
R^\circ &= \text{proper transform of} \ \bar{R} \\ 
E^\circ &= \text{proper transform of} \ \bar{E} \\
L_i &= \text{proper transform of} \ \bar{L}_i
\end{align*}
For \(q \in \Delta\), write \(F_q\) for the exceptional divisor
over \(q\). Set
\(p_i = L_i \cap F_p\).

The pair \((S, E)\) consisting of a surface $S$ and divisor $E$ as  in Setup \ref{eq:S_sigma_Q_diagram} admits a degeneration to \((S^\circ, E^\circ \cup L_1 \cup L_2)\) as an abstract pair of a surface with a divisor. Under the complete linear series \(|\O_{S^\circ}(n, n)(-\sum_{q \in \Delta} F_q)|\), the lines \(L_i\) get contracted to the points \(x_i\); thus, in \(\pp^{2n+1}\), the curve \(E\) limits to \(E^\circ\) embedded in \(\pp^{2n+1}\) as an elliptic normal curve. In the limit, the linear series corresponding to the maps \(f_i\) acquire basepoints at \(x_i\) on \(E^\circ\). Blowing up at \(x_1\) and \(x_2\) to extend the maps across the central fiber, the limiting maps have degree \(n\) on \(E^\circ\) and \(1\) on the corresponding exceptional lines \(L_i\). Equivalently, they are induced by projection of \(E^\circ \cup L_1 \cup L_2\) onto the two \(\pp^1\) factors.

We now show that  \(|\O_{S^\circ}(n, n)(-\sum_{q \in \Delta} F_q)|\) is basepoint free. Let $\bar{\Delta}$ denote the nodes of $\bar{E}$. By the discussion in \S \ref{sec:HN}, the linear series \(|\O_{\bar{S}}(n-1, n-1)(-\sum_{q \in \bar{\Delta}} F_q)|\)  is basepoint free. Pulling back to $S^\circ$, we conclude that \(|\O_{S^{\circ}}(n-1, n-1)(-\sum_{q \in \bar{\Delta}} F_q)|\)  is basepoint free. Multiplying by the equations of the lines $L_1$ and $L_2$, we see that any basepoints of  \(|\O_{S^\circ}(n, n)(-\sum_{q \in \Delta} F_q)|\) must lie on  the lines $L_i$. Since $\O_{S^\circ}(n, n)(-\sum_{q \in \Delta} F_q)|_{L_i}$ has degree zero, if there is a base point on $L_i$, then the linear series must identically vanish on $L_i$. An easy dimension count rules this possibility out.

Under this complete linear series
\(|\O_{S^\circ}(n, n)(-\sum_{q \in \Delta} F_q)|\),
the image of \(R^\circ\) is of degree \(2n - 1\),
and \(F_p\) is mapped to the line which
meets \(E^\circ\) at \(x_1\) and \(x_2\) (and which also meets the other
component \(R^\circ\)).
The images of \(E^\circ\), \(R^\circ \cup F_p\), and \(S^\circ\), along with the cone \(Q^\circ\) over \(\bar{Q}\) with vertex \(\overline{x_1 x_2}\),
give a degeneration of \((E, R, S, Q)\) in our Setup \ref{eq:S_sigma_Q_diagram}  as subschemes of $\pp^{2n+1}$.

Consider \((E, R, S, Q)\) limiting to \((E^\circ, R^\circ \cup F_p, S^\circ, Q^\circ)\).
The above description shows that
\(x_1\) and \(x_2\) are limits of points \(\hat{x}_1, \hat{x}_2 \in \Gamma \colonequals E \cap R\). Write \(\Gamma_- = \Gamma \setminus \{\hat{x}_1, \hat{x}_2\}\).
The limit of \(\Gamma_-\) is identified with \(\bar{\Gamma}\).
Our next task is to determine the flat limit of the bundles
\[N_{E/Q}[2\Gamma_- + \hat{x}_1 + \hat{x}_2 \posmod N_{E/S}] = N_{E/Q}[2\Gamma_-\posmod N_{E/S}][\hat{x}_1 + \hat{x}_2 \posmodalong R].\]
This is subtle precisely because \(E^\circ\) passes through
\(Q^\circ_\text{sing}\) (in particular the flat limit is not just \(N_{E^\circ/Q^\circ}[2\bar{\Gamma} \posmod N_{E^\circ / S^\circ}][x_1 + x_2 \posmodalong R^\circ]\)).
To do this, define 
\[B \colonequals \Bl_{Q_\sing} Q \qquad \text{and} \qquad B^\circ \colonequals \Bl_{Q_\sing^\circ} Q^\circ.\]
Explicitly, \(B\) is the graph of the
rational map \(Q \dashrightarrow \pp^1 \times \pp^1\) given by projection from \(Q_\sing\), and similarly for \(B^\circ\).  As in Setup \ref{eq:S_sigma_Q_diagram}, the composition of the map \(S \to Q\) with this projection is the blowup map \(S \to \pp^1 \times \pp^1\), and similarly for \(S^\circ\). The exceptional divisor of \(B\) is isomorphic to
\[[Q_\sing \simeq \pp^{2n - 3}] \times \pp^1 \times \pp^1,\]
and similarly for \(B^\circ\).

The line \(\bar{x_1x_2}\) naturally embeds in \(Q^\circ_\sing\) (coinciding with the image of \(F_p\)) and the lines \(L_i\) naturally embed in \(\pp^1 \times \pp^1\).  The flat limit of the (proper transform of) \(E\) in \(B\) specializes to the curve \(E^\circ \cup L_1 \cup L_2\) in \(B^\circ\).  In this limit, the points \(\hat{x}_1\) and \(\hat{x}_2\) limit to \(p_1 \in L_1\) and \(p_2 \in L_2\), where, as above, \(p_i = L_i \cap F_p\).  This setup is illustrated in the following picture.  The points in the limit of \(\Gamma\) (namely, \(\bar{\Gamma} \cup \{p_1, p_2\}\)) are circled.
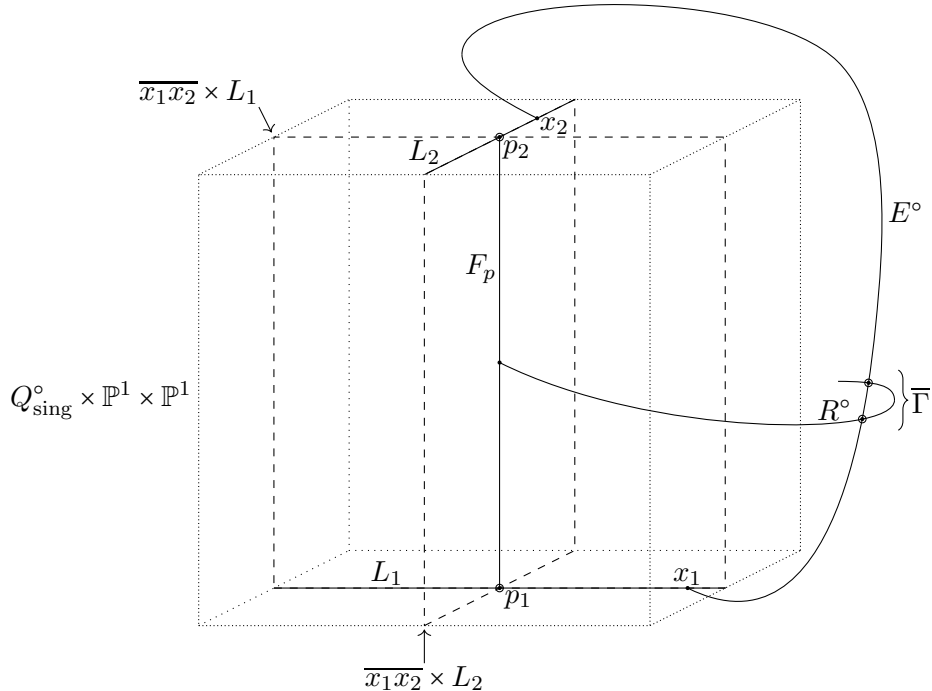
\begin{center}
\begin{tikzpicture}
\draw[densely dotted] (0, 0) -- (6, 0) -- (6, 6) -- (0, 6) -- (0, 0);
\draw[densely dotted] (8, 1) -- (8, 7) -- (2, 7);
\draw[dotted] (8, 1) -- (2, 1) -- (2, 7);
\draw[densely dotted] (0, 6) -- (2, 7);
\draw[densely dotted] (6, 6) -- (8, 7);
\draw[dotted] (0, 0) -- (2, 1);
\draw[densely dotted] (6, 0) -- (8, 1);
\draw (1, 0.5) -- (7, 0.5);
\draw (4, 0.5) -- (4, 6.5);
\draw (3, 6) -- (5, 7);
\draw (6.5, 0.5) .. controls (8.5, -0.5) and (8.9, 3) .. (9, 4);
\draw (4, 3.5) .. controls (6, 2.5) and (9.25, 2.5) .. (9.25, 3);
\draw (9.25, 3) .. controls (9.25, 3.2) and (9, 3.25) .. (8.5, 3.25);
\draw (4.5, 6.75) .. controls (4.0, 7) and (3.5, 7.3) .. (3.5, 7.6);
\draw (3.5, 7.6) .. controls (3.5, 8.5) and (7.5, 8.5) .. (8.5, 7.5);
\draw (8.5, 7.5) .. controls (9.25, 6.75) and (9.1, 5) .. (9, 4);
\draw (-1.3, 3) node{\(Q^\circ_\sing \times \pp^1 \times \pp^1\)};
\draw (2.5, 0.7) node{\(L_1\)};
\draw (3, 6.3) node{\(L_2\)};
\filldraw (4, 3.5) circle[radius=0.02];
\filldraw (6.5, 0.5) circle[radius=0.02];
\filldraw (4.5, 6.75) circle[radius=0.02];
\filldraw (4, 0.5) circle[radius=0.02];
\filldraw (4, 6.5) circle[radius=0.02];
\filldraw (8.822, 2.748) circle[radius=0.02];
\filldraw (8.905, 3.23) circle[radius=0.02];
\draw (4, 0.5) circle[radius=0.05];
\draw (4, 6.5) circle[radius=0.05];
\draw (8.822, 2.748) circle[radius=0.05];
\draw (8.905, 3.23) circle[radius=0.05];
\draw (3.75, 4.75) node{\(F_p\)};
\draw (8.45, 2.87) node{\(R^\circ\)};
\draw (9.4, 5.5) node{\(E^\circ\)};
\draw (4.23, 6.33) node{\(p_2\)};
\draw (4.73, 6.65) node{\(x_2\)};
\draw (4.25, 0.35) node{\(p_1\)};
\draw (6.5, 0.65) node{\(x_1\)};
\draw [decorate, decoration={brace, mirror, amplitude=0.75ex}] (9.3, 2.6) -- (9.3, 3.4);
\draw (9.6, 3) node{\(\bar{\Gamma}\)};
\draw[dashed] (1, 0.5) -- (7, 0.5) -- (7, 6.5) -- (1, 6.5) -- (1, 0.5);
\draw[dashed] (3, 0) -- (5, 1) -- (5, 7) -- (3, 6) -- (3, 0);
\draw (3, -0.7) node{\(\overline{x_1 x_2} \times L_2\)};
\draw (0, 7.1) node{\(\overline{x_1 x_2} \times L_1\)};
\draw[->] (0.8, 6.9) -- (0.97, 6.56);
\draw[->] (3, -0.5) -- (3, -0.05);
\end{tikzpicture}
\end{center}

\begin{lem}\label{lem:flat_limit_Ncirc}
Let \((E, R, S, Q)\) be a general instance of Setup \ref{eq:S_sigma_Q_diagram} in \(\pp^{2n+1}\).  Then 
\[N_{E/Q}[2\Gamma_-\posmod N_{E/S}][\hat{x}_1 + \hat{x}_2 \posmodalong R] \] 
admits a specialization to
\[N^\circ \colonequals  N_{E^\circ / B^\circ}[2\bar{\Gamma}
 \posmod N_{E^\circ/S^\circ}][x_1 \posmodalong \bar{x_1 x_2} \times L_1][x_2 \posmodalong \bar{x_1x_2} \times L_2].\]
\end{lem}
\begin{proof}
Since \(E\) does not meet \(Q_\sing\), we have
\[N_{E/Q}[2\Gamma_-\posmod N_{E/S}][\hat{x}_1 + \hat{x}_2 \posmodalong R] \simeq  N_{E/B}[2\Gamma_-\posmod N_{E/S}][\hat{x}_1 + \hat{x}_2 \posmodalong R].\]
This bundle fits into a flat family \(\N\) whose central fiber is
\[N \colonequals \N|_0 = N_{E^\circ \cup L_1 \cup L_2 / B^\circ}[2 \bar{\Gamma} \posmod N_{E^\circ/S^\circ}][p_1 + p_2\posmodalong  F_p].\]

The \(L_i\) are lines in the exceptional divisor of types
\((0, 1, 0)\) and \((0, 0, 1)\), respectively. In particular,
their normal bundles in the exceptional divisor are trivial,
and so their normal bundles in
\(B^\circ\) are \(\O_{L_i}^{\oplus(r-2)} \oplus \O_{L_i}(-1)\).
The restriction
\[N|_{L_i} \simeq N_{L_i/B^\circ}[x_i \posmodalong E^\circ][p_i \posmodalong F_p]\]
is obtained by making two positive modifications.
Since the restriction of the projection  from \(Q^\circ_\sing\) to \(E^\circ\) has degree \(2n\) (i.e., equal to the degree of \(\bar{E}\)) by construction, \(E^\circ\) must meet \(Q_\sing\) at \(x_1\) and \(x_2\), both with multiplicity \(1\). In the blowup, \(E^\circ\) is therefore transverse to the exceptional divisor at the \(x_i\),
so the positive modification at \(x_i\) is transverse to \(\O_{L_i}^{\oplus(r-2)}\).
Therefore
\[N|_{L_i} \simeq \O_{L_i}^{\oplus(r-2)} \oplus \O_{L_i}(1).\]

To identify the positive subbundle, note that there is a unique
subbundle of \(N|_{L_i}\) that is isomorphic to \(\O_{L_i}(1)\),
and that one such subbundle is \(N_{L_i / \bar{x_1 x_2} \times L_i}(p_i)\).

Consider the modification
\[\N' \colonequals \N[L_1 \posmod N_{L_1 / \bar{x_1 x_2} \times L_1}(p_1)][L_2 \posmod N_{L_2 / \bar{x_1 x_2} \times L_2}(p_2)].\]
Away from the central fiber, we have \(\N' \simeq \N\).
The central fiber \(\N'|_0\) therefore gives
another flat limit of the bundle \(N_{E/Q}[2\Gamma_-\posmod N_{E/S}][\hat{x}_1 + \hat{x}_2 \posmodalong R] \).
But by construction, \(\N'|_0\)
has trivial restriction to \(L_1\) and \(L_2\) by \eqref{eq:direct_on_D}.
Blowing down \(L_1\) and \(L_2\),
we conclude that
a flat limit of the bundles \(N_{E/Q}[2\Gamma_-\posmod N_{E/S}][\hat{x}_1 + \hat{x}_2 \posmodalong R] \) is therefore
\[N^\circ = \N'|_{E^\circ} \simeq N_{E^\circ / B^\circ}[2\bar{\Gamma} \posmod N_{E^\circ/S^\circ}][x_1 \posmodalong \bar{x_1 x_2} \times L_1][x_2 \posmodalong \bar{x_1x_2} \times L_2].\qedhere\]
\end{proof}

Our final goal is to relate this to projection from the line \(\bar{x_1 x_2}\).  By construction, this projection map sends \((E^\circ, R^\circ, S^\circ, Q^\circ)\) in \(\pp^{2n+1}\) to \((\bar{E}, \bar{R}, \bar{S}, \bar{Q})\) in \(\pp^{2n-1}\).
We accomplish this by rewriting \(N^\circ\) in terms of the normal
bundle of the proper transform of \(E^\circ\) in
\[B_-^\circ \colonequals \Bl_{\bar{x_1 x_2}} Q^\circ.\]
Because \(\bar{x_1x_2} \subset Q^\circ_\sing\), there is a natural map from the exceptional divisor of \(B^\circ_-\) to the exceptional divisor of \(B^\circ\).
Write \(M_i\) for the preimage of \(\bar{x_1 x_2} \times L_i\) in the exceptional divisor of \(B_-^\circ\).
Then by \eqref{eq:blowup_modification} in Section~\ref{sec-prelim-elementary}, we have
\[N^\circ \simeq N_{E^\circ / B_-^\circ}[2\bar{\Gamma} \posmod N_{E^\circ/S^\circ}] [x_1 \posmodalong M_1][x_2 \posmodalong M_2].\]
Explicitly, the exceptional divisor of \(B^\circ_-\)
is isomorphic to \(\bar{x_1 x_2} \times \bar{Q}\),
with \(M_i = \bar{x_1 x_2} \times \bar{M}_i\),
where the \(\bar{M}_i\) are the
 \((2n - 3)\)-planes of the rulings
of \(\bar{Q}\) corresponding to \(L_i\).

Note that \(\bar{x_1 x_2} \times p\) is contained
in \(M_1\) and \(M_2\), and is contracted to the point \(p \in \bar{Q}\)
under projection.
Moreover, \(\bar{M}_i\)
is transverse (not just quasi-transverse!) to \(\bar{E}\) at \(x_i\).
Projection from \(\bar{x_1 x_2}\) therefore induces the exact sequence
\begin{equation}\label{eq:exact_Qsing}
0 \to \O_{E^\circ}(1)(x_1 + x_2)^{\oplus 2} \to N^\circ \to N_{\bar{E} / \bar{Q}}[2 \bar{\Gamma} \posmod N_{\bar{E}/\bar{S}}](x_1 + x_2) \to 0.
\end{equation}

\section{Completing the proof in even genus}\label{sec:even_g_numerics}

Let \(g = 2n+2\) be even.  We consider the degenerate canonical curve \(E \cup R \subset \pp^{2n+1}\) introduced in Section \ref{sec:our_degen}.
In this section, we leverage the geometric description of the HN-filtration of \(N_{E\cup R}|_R\) given in Section \ref{sec:HN} and the semistability of \(N_{E \cup R}|_E\) proved in Section \ref{sec:restriction_to_E} to prove that the normal bundle of a general canonical curve of even genus is semistable.  

Let \(S\), \(\Sigma_1\), \(\Sigma_2\), and \(Q\) be as in
Setup~\ref{eq:S_sigma_Q_diagram}. We first reduce to proving a bound on the slopes of certain subbundles of \(N_{E \cup R/Q}|_E\).

\begin{cond}\label{cond:numerical}
For a general \(E \cup R \subset \pp^{2n+1}\), every subbundle \(F \subseteq N_{E \cup R/Q}|_E\) with
\[N_{E \cup R/S}|_\Gamma \subset F|_\Gamma\]
satisfies
\[\mu(F) \leq 2n+5 + \frac{3}{n} - \frac{1}{\rk F}.\]
\end{cond}

\begin{cond}\label{cond:numerical-double}
For a general \(E \cup R \subset \pp^{2n+1}\), every subbundle \(F \subseteq N_{E/Q}[2\Gamma \posmod N_{E/S}]\)
satisfies
\[\mu(F) \leq 2n+5 + \frac{3}{n} + \frac{2n + 1}{\rk F}.\]
\end{cond}

\begin{lem}
Condition~\ref{cond:numerical-double} implies Condition~\ref{cond:numerical}.
\end{lem}
\begin{proof}
If \(F\) is a subbundle of \(N_{E \cup R/Q}|_E = N_{E/Q}[\Gamma \posmod N_{E/S}]\) with \(N_{E \cup R/S}|_\Gamma \subset F|_\Gamma\), then the modification \(F[\Gamma \posmod N_{E/S}]\) is a subbundle of \(N_{E/Q}[2\Gamma \posmod N_{E/S}]\) with
\[\mu(F[\Gamma \posmod N_{E/S}]) = \mu(F) + \frac{2n+2}{\rk{F}}.\qedhere\]
\end{proof}

\begin{prop}\label{prop:bound_subs}
Suppose that Condition \ref{cond:numerical} is satisfied.
Then \(N_{E \cup R}\) is semistable.
\end{prop}
\begin{proof}
Let \(\nu \colon E \sqcup R \to E \cup R\) denote the normalization,
and \(G \subseteq \nu^* N_{E \cup R}\) be any subbundle.  By Lemma~\ref{lem:N_restrict_R} and Theorem~\ref{thm:restriction_ss},
we have
\[\mu(G|_R) \leq 2n+3 + \frac{1}{\rk G} \qand \mu(G|_E) \leq \mu(N_{E \cup R}|_E) = 2n+5 + \frac{3}{n}.\]
Combining these, we have 
\[\mu^{\text{adj}}(G) \leq \mu(G) = \mu(G|_R) + \mu(G|_E) \leq 4n + 8 + \frac{3}{n} + \frac{1}{\rk G},\]
with the stronger bound
\[\mu^{\text{adj}}(G) \leq 4n + 8 + \frac{3}{n} = \mu(N_{E \cup R})\]
unless \(G\) is actually a subbundle of \(N_{E \cup R}\) and \(N_{E \cup R/S}|_R \subset G|_R \subset N_{E\cup R/Q}|_R\). In other words, we are immediately done unless \(G \) is a subbundle of \(N_{E \cup R}\) and \(G|_R\) contains the positive factor \(\O_{\pp^1}(g+2)\) and is contained in next piece of the HN-filtration \(\O_{\pp^1}(g+2) \oplus \O_{\pp^1}(g+1)^{\oplus (g-4)}\).  We therefore assume that these hold.  The restriction \(G|_E\) is thus a subbundle of \(N_{E \cup R}|_E\)  with \(N_{E \cup R/S}|_\Gamma \subset G|_\Gamma\).  Write \(G'\) for the kernel of the map from \(G|_E\) to \(N_{Q}|_E\):
\begin{center}
\begin{tikzcd}
0 \arrow[r] & N_{E \cup R/Q}|_E \arrow[r] & N_{E \cup R}|_E \arrow[r] & N_Q|_E \simeq \O_E(2) \arrow[r] & 0\\
& \ker(G|_E \to N_Q|_E) = G' \arrow[r] \arrow[u, hook]& G|_E \arrow[u, hook] \arrow[ur]
\end{tikzcd}
\end{center}
If \(G|_E \neq G'\), then the map \(G|_E/G' \to \O_E(2)\) factors through \(\O_E(2)(-\Gamma)\), which is stable of slope \(2n+2\).  On the other hand, the kernel \(G' \subseteq N_{E \cup R/Q}|_E\) has slope
\[\mu(G') \leq 2n+5 + \frac{3}{n} - \frac{1}{\rk G'} \leq 2n+5 + \frac{3}{n} - \frac{1}{\rk G},\]
by Condition~\ref{cond:numerical}.  Thus \(G|_E\) also has slope bounded by \(2n+5 + \frac{3}{n} - \frac{1}{\rk G}\).  Hence
\[\mu(G) = \mu(G|_R) + \mu(G|_E) \leq \left(2n+3 + \frac{1}{\rk G}\right) + \left(2n+5 + \frac{3}{n} - \frac{1}{\rk G} \right) = \mu(N_{E \cup R}),\]
and \(N_{E \cup R}\) is semistable.
\end{proof}

Our goal is therefore to prove that Condition \ref{cond:numerical} holds for all \(n \geq 3\).   In fact, we will prove that Condition \ref{cond:numerical-double} holds for all \(n \geq 3\), since this implies that Condition \ref{cond:numerical} holds.
While Condition \ref{cond:numerical-double} is stated for \emph{all} subbundles of \(N_{E/Q}[2\Gamma \posmod N_{E/S}]\), it suffices to check the slope bound for the finitely many Harder--Narasimhan pieces.

\begin{lem}\label{lem:convex}
Let \(N\) be a vector bundle on an irreducible curve \(C\) with HN-filtration
\[0 \subset V_1 \subset V_2 \subset \cdots \subset V_m = N.\]
Let \(B(r,d)\) be any (affine) linear function whose coefficient
of \(d\) is nonnegative. If \(B(0,0) \leq 0\) and \(B(\rk V_i, \deg V_i) \leq 0\) for all \(i\), then \(B(\rk F, \deg F) \leq 0\) for all subbundles \(F \subset N\).
\end{lem}
\begin{proof}
The points
\[\left\{(0, -\infty), (0,0), (\rk V_1, \deg V_1), \dots, (\rk V_m, \deg V_m)\right\}\]
form the vertices of a convex polygon in the \((r,d)\) plane.  For any subbundle \(F \subseteq N\), the pair \((\rk F, \deg F)\) is in this polygon.  The assumption that \(B(r, d) \leq 0\) for all vertices implies that it is also true for any point of the convex polygon.
\end{proof}

\begin{cor}\label{cor:check_HN}
Suppose that for each HN-piece \(V\) of \(N_{E/Q}[2\Gamma \posmod N_{E/S}]\) we have
\[\mu(V) \leq 2n + 5  + \frac{3}{n} + \frac{2n+1}{\rk V}.\]
Then Condition \ref{cond:numerical-double} holds.
\end{cor}
\begin{proof}
Apply Lemma \ref{lem:convex} with 
\[B(r, d) = d - \left(2n + 5 + \frac{3}{n}\right)r - (2n+1).\qedhere\]
\end{proof}

The final input is the specialization of \((E, R, S, Q)\) constructed in Section \ref{sec:Qsing}, giving rise to the exact sequence \eqref{eq:exact_Qsing}.
Using this, we will prove the following numerical proposition, which is the heart of our inductive proof.

\begin{prop}
Let \(n>3\). If Condition \ref{cond:numerical-double} holds in \(\pp^{2n-1}\), then it holds in \(\pp^{2n+1}\).
\end{prop}
\begin{proof}
We will use the notation and results of Section \ref{sec:Qsing}.  In particular, let \((\bar{E}, \bar{R}, \bar{S}, \bar{Q})\) be a general instance of Setup \ref{eq:S_sigma_Q_diagram} in \(\pp^{2n-1}\).  Let \(x_1, x_2\) be points on \(\bar{E}\) such that \(\bar{f}_1(x_1) = \bar{f}_2(x_2)\).  Then in Section \ref{sec:Qsing} we constructed a specialization \((E^\circ, R^\circ \cup F_p, S^\circ, Q^\circ)\) of a general instance \((E, R, S, Q)\) of Setup \ref{eq:S_sigma_Q_diagram} in \(\pp^{2n+1}\), such that \(E^\circ\) meets \(Q^\circ_{\text{sing}}\) in the points \(x_1, x_2\).  We write \(\hat{x}_1, \hat{x}_2\) for points on \(E\) limiting to \(x_1, x_2\).

Applying Corollary \ref{cor:check_HN}, it suffices to check that Condition \ref{cond:numerical-double} holds for each piece
of the HN-filtration of \(N_{E/Q}[2\Gamma \posmod N_{E/S}]\).  Since the HN-pieces are natural, their degrees are multiples of \(2n+2\) by Lemma \ref{lem:blackmagic-natural}.
Let \(F_0 \subset N_{E/Q}[2\Gamma \posmod N_{E/S}]\) be any such subbundle of rank \(r\) and degree a multiple of \(2n+2\). 
Let \(F \subset N_{E/Q}[2\Gamma - \hat{x}_1 - \hat{x}_2 \posmod N_{E/S}]\) be the intersection
of \(F_0\) with \(N_{E/Q}[2\Gamma - \hat{x}_1 - \hat{x}_2 \posmod N_{E/S}]\).
We have
\[\mu(F_0) \leq \mu(F) + \frac{2}{\rk F},\]
so it suffices to show
\[\mu(F) \leq 2n + 5 + \frac{3}{n} + \frac{2n - 1}{\rk F}.\]

We now utilize the specialization constructed in Section \ref{sec:Qsing}.  Write \(N^\circ\) for the bundle appearing in Lemma \ref{lem:flat_limit_Ncirc}, which is a flat limit of the bundle \(N_{E/Q}[2\Gamma - \hat{x}_1 - \hat{x}_2 \posmod N_{E/S}]\).  This bundle sits in the exact sequence
\[ 0 \to \O_{E^\circ}(1)(x_1 + x_2)^{\oplus 2} \to N^\circ  \stackrel{\phi}{\longrightarrow} N_{\bar{E} / \bar{Q}}[2 \bar{\Gamma} \posmod N_{\bar{E}/\bar{S}}](x_1 + x_2) \to 0.\]
Let \(F^\circ \subseteq N^\circ\) be the saturation of the flat limit of \(F\).  Then \(\rk F^\circ = \rk F = r\) and \(\deg F^\circ \geq \deg F\).

\medskip
\noindent
\textbf{\boldmath Case 1: \(F^\circ\) intersects \(\ker \phi\) nontrivially.} We obtain an exact sequence
\[0 \to F' \to F^\circ \to F'' \to 0,\] 
where \(F' \subset \O_{E^\circ}(1)(x_1 + x_2)^{\oplus 2}\) and \(F'' \subset N_{\overline{E}/\overline{Q}}[2\overline{\Gamma}\posmod N_{\overline{E}/\overline{S}}](x_1 + x_2)\).
Since \(\O_{E^\circ}(1)(x_1+x_2)^{\oplus 2}\) is semistable, we have that
\[ \mu(F') \leq \mu (\O_{E^\circ}(1)(x_1 + x_2)^{\oplus 2}) = 2n+4.\]
By our inductive hypothesis, we have that 
\[\mu(F'') \leq 2n+5+ \frac{3}{n-1} + \frac{2n-1}{\rk F''}.\]  
We conclude that 
\begin{equation}\label{eqn-muF}
\mu(F^\circ) \leq \frac{1}{r} \bigg(2n+4\bigg) + \frac{r-1}{r} \bigg(2n+5 + \frac{3}{n-1} + \frac{2n-1}{r-1}\bigg) = 2n + 5 + \frac{3r-2-n}{r(n-1)} + \frac{2n-1}{r}.
\end{equation}
If \(n \geq 3\), then \(6n-3 \leq n^2 + 2n\). Since \(r \leq \rk(N_{E/Q}[2\Gamma \posmod N_{E/S}]) =2n-1\), we have the inequality  \(3r \leq n^2 + 2n\), which implies that 
\[\frac{3r-2-n}{r(n-1)} \leq \frac{3}{n} \quad \mbox{if}  \quad n \geq 3.\] 
Substituting into \eqref{eqn-muF}, we get
\[\mu(F) \leq \mu(F^\circ) \leq 2n+5 + \frac{3}{n} + \frac{2n-1}{r},\]
which proves the proposition when \(F\) intersects the kernel of \(\phi\) nontrivially.

\medskip
\noindent
\textbf{\boldmath Case 2: \(F^\circ\) is isomorphic to its image under \(\phi\).} Identifying \(F^\circ\) with its image under \(\phi\), we have
\[F^\circ \subset N_{\overline{E}/\overline{Q}}[2\overline{\Gamma}\posmod N_{\overline{E}/\overline{S}}](x_1 + x_2) \qand r \leq 2n-3.\] 
By the inductive hypothesis, we have that 
\[\mu(F^\circ) \leq 2n+5+ \frac{3}{n-1} + \frac{2n - 1}{r},\]
and hence for the general fiber we have
\[ \deg(F_0) \leq \deg(F) + 2 \leq  \deg(F^\circ) + 2 \leq (2n+5) r + \frac{3r}{n-1} + 2n + 1.\]
We would instead like to show the stronger inequality  
\begin{equation}\label{eqn-desired}
\deg(F_0) \leq (2n+5) r + \frac{3r}{n} + 2n + 1.
\end{equation}

To do this, we will use the fact that naturality of the HN-pieces implies that \(2n+2\) divides \(\deg(F_0)\).
If there are no integers \(k\) satisfying the inequality
\begin{equation}\label{eqn-3rn}
\frac{3r}{n} < k \leq \frac{3r}{n-1},
\end{equation}
then \eqref{eqn-desired} holds, as \(\deg(F_0)\) is an integer. Hence, we assume that there is an integer \(k\) satisfying (\ref{eqn-3rn}).

First, suppose \(n \geq 6\). We claim that the width of the interval
\eqref{eqn-3rn} is strictly less than \(1\).
Indeed, since \(r \leq 2n - 3\) and \(n \geq 6\),
\[\frac{3r}{n-1}- \frac{3r}{n}= \frac{3r}{n^2-n} \leq \frac{6n-9}{n^2-n} < 1.\]
If \eqref{eqn-desired} does not hold, then 
\[3r+k  -1 \equiv (2n+5)r + k +2n+1 = \deg(F_0) \equiv 0 \pmod{2n + 2},\]
i.e., we may write \(3r + k - 1 = (2n + 2)\ell\) for an integer \(\ell\).
Plugging this back into \eqref{eqn-3rn}, and subtracting \(2\ell\)
from each term, we obtain
\[\frac{2\ell-k+1}{n} < k-2\ell \leq \frac{4\ell -k +1}{n-1}.\]
If \(k-2\ell \leq 0\), the left inequality is violated. If \(k-2\ell \geq 1\), then the left fraction is nonpositive, which contradicts our observation that the width of this interval is strictly less than \(1\).

For \(4 \leq n \leq 5\), we complete the proof by checking directly that
there are no integers \(k\) satisfying
the conditions
\begin{equation}
\frac{3 r}{n} < k \leq \frac{3 r}{n-1}, \quad 0< r \leq 2n  \quad \mbox{and} \quad 3 r + k -1 \equiv 0 \pmod{2n+2}. \label{wwww}
\end{equation}
The following three tables summarize the possible values of \(k\) for each value of \(r\) and compute  \(3 r + k -1 \pmod{2n+2}\) in the two cases \(n=4\) and \(n=5\), respectively.

\bigskip

\(n=4 \quad\)
\begin{tabular}{|c|c|c|c|c|c|c|}  \hline
 \(r\)   & 1& 2& 3& 4& 5& 6 \\ \hline
 \(k\)  & 1& 2& 3 & 4 & 4  or 5 & 5 or 6 \\ \hline 
 \(3r+k-1 \pmod{10}\) & 3 & 7 & 1 & 5 & 8 or  9 & 2 or 3 \\ \hline
\end{tabular}
\bigskip

\(n=5 \quad\)
\begin{tabular}{|c|c|c|c|c|c|c|c|c|}  \hline
\(r\)    & 1& 2& 3& 4& 5& 6 & 7 & 8  \\ \hline 
\(k\)    & \(\emptyset\) & \(\emptyset\) & 2 & 3 & \(\emptyset\) & 4 & 5 &  5 or 6 \\ \hline 
 \(3r+k-1 \pmod{12}\) &  &  & 10 & 2 &  & 9 & 1 & 4 or 5  \\ \hline
\end{tabular}
\bigskip

\noindent
We see that when \(n>3\), there are no integers \(k\) satisfying \eqref{wwww}.
\end{proof}

\noindent
To finish, it suffices to deal with the base case:

\begin{prop}
Condition \ref{cond:numerical-double} holds in \(\pp^7\).
\end{prop}
\begin{proof}
In this case \(n=3\) and we want that every subbundle \(F\) of \(N_{E/Q}[2\Gamma \posmod N_{E/S}]\) has slope at most \(12 + \frac{7}{\rk F}\).   To prove this, we use the normal bundle exact sequence for \(E \subset S \subset Q\).  
Since \(S\) is the complete intersection of \(\Sigma_1\) and \(\Sigma_2\) in \(Q\), we have that
\[N_{S/Q}|_E \simeq N_{S/\Sigma_1}|_E \oplus N_{S/\Sigma_2}|_E.\]
Since \(4\) is not divisible by \(8\), by Lemma~\ref{lem:blackmagic-natural}, the degree \(4\) maps giving rise to the scrolls \(\Sigma_1\) and \(\Sigma_2\) are exchanged by monodromy.   Hence the two scrolls are exchanged by monodromy, and therefore the two bundles \(N_{S/\Sigma_i}|_E\) have degree \(24\) and the same profile of Jordan--H\"{o}lder factors.  We will first show that \(N_{S/\Sigma_i}|_E\) is semistable of slope \(12\).

The bundle \(N_E[\Gamma \posmod N_{E/S}] = N_E[\posmodalong R]\) has slope \(12\) and satisfies interpolation by \cite{lv22} (in the language of that paper, this is the inductive hypothesis \(I(8, 1, 7, 0, 1)\) and the tuple \((8, 1, 7, 0, 1)\) is good).  Because any bundle with integral slope that satisfies interpolation is semistable, this bundle is semistable (see, for example, \cite[Remark 1.6]{v18}.) Consider the normal bundle exact sequence
\[0 \to N_{E/S}(\Gamma) \to N_{E/Q}[\Gamma \to N_{E/S}] \to \left[N_{S/Q}|_E\simeq N_{S/\Sigma_1}|_E \oplus N_{S/\Sigma_2}|_E\right] \to 0.\]
The line subbundle \(N_{E/S}(\Gamma)\) has degree \(8\).
First suppose that one (and hence both) of \(N_{S/\Sigma_i}|_E\) had a line subbundle of slope at least \(15\).  Then the full preimage in \(N_{E/Q}[\Gamma \posmod N_{E/S}]\) would be a bundle of slope at least \(38/3 = 12 + 2/3\).  Since this is also a subbundle of \(N_{E}[\Gamma \posmod N_{E/S}]\), it contradicts the semistability of \(N_{E}[\Gamma \posmod N_{E/S}]\).  Hence every line subbundle of \(N_{S/\Sigma_i}|_E\) is of degree at most \(14\).  It suffices, therefore, to rule out the possibility that \(N_{S/\Sigma_i}|_E\) is a direct sum of line bundles of degrees \(14, 10\) or \(13, 11\).
In either of these cases, the sum of the two positive subbundles
would be of degree \(28\) or \(26\).
Since \(8\) does not divide \(28\) or \(26\), this is impossible by Lemma~\ref{lem:blackmagic-natural}.
Hence \(N_{S/\Sigma_i}|_E\) is semistable.

We now turn to the normal bundle exact sequence involving the double modification
\[0 \to N_{E/S}(2\Gamma) \to N_{E/Q}[2\Gamma \posmod N_{E/S}] \to \left[N_{S/Q}|_E\simeq N_{S/\Sigma_1}|_E \oplus N_{S/\Sigma_2}|_E\right] \to 0,\]
and consider how \(F\) sits with respect to this sequence.  If \(F\) does not contain \(N_{E/S}(2\Gamma)\), then \(\mu(F) \leq 12 \leq 12 + \frac{7}{\rk F}\).  If \(F\) contains \(N_{E/S}(2\Gamma)\), then
\[\mu(F) \leq 16 \left( \frac{1}{\rk F} \right) + 12 \left( \frac{\rk F - 1}{\rk F} \right) \leq 12 + \frac{4}{\rk F} \leq 12 + \frac{7}{\rk F}. \qedhere\]
\end{proof}

\bibliographystyle{plain}

\end{document}